\documentclass[10pt,twoside]{amsart}
\usepackage{}

\usepackage{amsthm}
\usepackage{amsmath}
\usepackage{amssymb}
\usepackage{amsfonts}
\usepackage{latexsym}
\usepackage[all]{xy} \SelectTips{eu}{}
\usepackage{hyperref}
\usepackage{xcolor}
\usepackage{eucal}

\newtheorem{intthm}{Theorem}[]

\newtheorem*{intque*}{Question}
\newtheorem*{intexa*}{Example}
\newcommand{\numberseries}{\bfseries}   
\newlength{\thmtopspace}                
\newlength{\thmbotspace}                
\newlength{\thmheadspace}               
\newlength{\thmindent}                  
\newtheoremstyle{bfupright head,slanted body}
                {\thmtopspace}{\thmbotspace}
                {\slshape}{\thmindent}{\bfseries}{.}{\thmheadspace}
                {{\numberseries \thmnumber{#2\;}}\thmnote{#3}}

\newtheoremstyle{bfupright head,upright body}
                {\thmtopspace}{\thmbotspace}
                {\upshape}{\thmindent}{\bfseries}{.}{\thmheadspace}
                {{\numberseries \thmnumber{#2\;}}\thmnote{#3}}

\newtheoremstyle{fixed bf head,slanted body}
                {\thmtopspace}{\thmbotspace}{\slshape}
                {\thmindent}{\bfseries}{.}{\thmheadspace}
                {{\numberseries \thmnumber{#2\;}}\thmname{#1}\thmnote{ (#3)}}

\newtheoremstyle{fixed bf head,upright body}
                {\thmtopspace}{\thmbotspace}{\upshape}
                {\thmindent}{\bfseries}{.}{\thmheadspace}
                {{\numberseries \thmnumber{#2\;}}\thmname{#1}\thmnote{ (#3)}}

\newtheoremstyle{numbered paragraph}
                {\thmtopspace}{\thmbotspace}{\upshape}
                {\thmindent}{\upshape}{}{\thmheadspace}
                {{\numberseries \thmnumber{#2.}}}

\theoremstyle{bfupright head,slanted body}
\newtheorem{res}{}[section]             \newtheorem*{res*}{}

\theoremstyle{bfupright head,upright body}
\newtheorem{bfhpg}[res]{}               \newtheorem*{bfhpg*}{}

\theoremstyle{fixed bf head,slanted body}
\newtheorem{thm}[res]{Theorem}          \newtheorem*{thm*}{Theorem}
\newtheorem{prp}[res]{Proposition}      \newtheorem*{prp*}{Proposition}
        \newtheorem*{cor*}{Corollary}
\newtheorem{lem}[res]{Lemma}            \newtheorem*{lem*}{Lemma}
\theoremstyle{fixed bf head,upright body}
\newtheorem{dfn}[res]{Definition}       \newtheorem*{dfn*}{Definition}
\newtheorem{rmk}[res]{Remark}           \newtheorem*{rmk*}{Remark}
\newtheorem{exa}[res]{Example}           \newtheorem*{exa*}{Example}
           \newtheorem*{que*}{Question}
           \newtheorem*{fact*}{Fact}
           \newtheorem*{nota*}{Notation}
           \newtheorem*{setup*}{Setup}

\theoremstyle{numbered paragraph}

\newlength{\thmlistleft}        
\newlength{\thmlistright}       
\newlength{\thmlistpartopsep}   
\newlength{\thmlisttopsep}      
\newlength{\thmlistparsep}      
\newlength{\thmlistitemsep}     

\newcounter{eqc}
\newenvironment{eqc}{\begin{list}{\upshape (\textit{\roman{eqc}})}%
    {\usecounter{eqc}%
      \setlength{\leftmargin}{\thmlistleft}%
      \setlength{\labelwidth}{\thmlistleft}%
      \setlength{\rightmargin}{\thmlistright}%
      \setlength{\partopsep}{\thmlistpartopsep}%
      \setlength{\topsep}{\thmlisttopsep}%
      \setlength{\parsep}{\thmlistparsep}%
      \setlength{\itemsep}{\thmlistitemsep}}}%
  {\end{list}}%

\newcounter{prt}
\newenvironment{prt}{\begin{list}{\upshape (\alph{prt})}%
    {\usecounter{prt}%
      \setlength{\leftmargin}{\thmlistleft}%
      \setlength{\labelwidth}{\thmlistleft}%
      \setlength{\rightmargin}{\thmlistright}%
      \setlength{\partopsep}{\thmlistpartopsep}%
      \setlength{\topsep}{\thmlisttopsep}%
      \setlength{\parsep}{\thmlistparsep}%
      \setlength{\itemsep}{\thmlistitemsep}}}%
  {\end{list}}%

\newcounter{rqm}
  {\end{list}}%

\newenvironment{prf*}[1][Proof]{%
  \begin{proof}[\bf #1]
    \setcounter{equation}{0}
    }
  {\end{proof}
}

\newcommand{\pgref}[1]{\ref{#1}}

\renewcommand{\eqref}[1]{(\pgref{eq:#1})}

\newcommand{\eqclbl}[1]{{\upshape(\textit{#1})}}

\numberwithin{equation}{res}
\def\urltilda{\kern -.15em\lower .7ex\hbox{\~{}}\kern .04em}

\newcommand{\xra}[2][]{\xrightarrow[#1]{\:#2\:}}
\newcommand{\Ker}[1]{\nobreak{\operatorname{Ker}#1}}
\newcommand{\Coker}[1]{\nobreak{\operatorname{Coker}#1}}
\newcommand{\id}{\mathrm{id}}

\newcommand{\h}{\mathrm{H}}

\newcommand{\calC}{\mathcal{C}}
\newcommand{\calF}{\mathcal{F}}
\newcommand{\calW}{\mathcal{W}}

\newcommand{\sfA}{\mathsf{A}}
\newcommand{\sfC}{\mathsf{C}}
\newcommand{\sfE}{\mathsf{E}}
\newcommand{\sset}{\mathsf{sSet}}
\newcommand{\sfF}{\mathsf{F}}
\newcommand{\sfW}{\mathsf{W}}

\newcommand{\Mon}{\mathsf{Mon}}
\newcommand{\Epi}{\mathsf{Epi}}
\newcommand{\bcalC}{\widetilde{\mathcal{C}}}
\newcommand{\bcalF}{\widetilde{\mathcal{F}}}
\newcommand{\bsfC}{\widetilde{\mathsf{C}}}
\newcommand{\bsfF}{\widetilde{\mathsf{F}}}

\newcommand{\Set}{\mathsf{Set}}

\newcommand{\CH}{\mathsf{Ch}_{\geqslant0}(R)}

\begin{document}

\title[Compatible weak factorization systems and model structures]{Compatible weak factorization systems and model structures}

\author[Z.X. Di]{Zhenxing Di}
\address{Z.X. Di \ School of Mathematical Sciences, Huaqiao University, Quanzhou 362021, China}
\email{dizhenxing@163.com}

\author[L.P. Li]{Liping Li}
\address{L.P. Li \ Department of Mathematics, Hunan Normal University, Changsha 410081, China.}
\email{lipingli@hunnu.edu.cn}

\author[L. Liang]{Li Liang}
\address{L. Liang \ Department of Mathematics, Lanzhou Jiaotong University, Lanzhou 730070, China}
\email{lliangnju@gmail.com}
\urladdr{https://sites.google.com/site/lliangnju}

\thanks{Z.X. Di was partly supported by NSF of China (Grant No. 12471034),
the Scientific Research Funds of Huaqiao University (Grant No. 605-50Y22050) and
the Fujian Alliance Of Mathematics (Grant No. 2024SXLMMS04).
L.P. Li was partly supported by NSF of China (Grant No. 12171146).
L. Liang was partly supported by NSF of China (Grant No. 12271230). }

\date{\today}



\begin{abstract}
In this paper, the concept of compatible weak factorization systems in general categories is introduced as a counterpart of compatible complete cotorsion pairs in abelian categories. We describe a method to construct model structures on general categories via two compatible weak factorization systems satisfying certain conditions, and hence, generalize a very useful result by Gillespie for abelian model structures. As particular examples, we show that weak factorization systems associated to some classical model structures (for example, the Kan-Quillen model structure on $\sset$) satisfy these conditions.
\end{abstract}

\maketitle

\thispagestyle{empty}

\section*{Introduction}
\noindent
Throughout this paper, let $\sfE$ denote a category and $\sfA$ denote an abelian category.

To study various homotopy theories on $\sfE$ via a uniform and axiomatic approach, Quillen \cite{Qui67} introduced {\it model structures}, which are triples $(\calC, \calW, \calF)$ of classes of morphisms in $\sfE$ such that both $(\calC, \calW \cap \calF)$ and $(\calC \cap \calW, \calF)$ are weak factorization systems, and $\calW$ satisfies the 2-out-of-3 property. This notion plays a central role in modern homotopy theory, and people try to find various methods to construct model structures on $\sfE$.

For an abelian category $\sfA$, people are more interested in \emph{abelian model structures} which is compatible with the abelian structure of $\sfA$ (see Hovey \cite{Ho02}), that is, calibrations coincide with monomorphisms with cofibrant cokernels and fibrations coincide with epimorphisms with fibrant kernels. The celebrated Hovey's correspondence provides a bijective correspondence between abelian model structures on $\sfA$ and Hovey triples, that is, triples $(\sfC, \sfW, \sfF)$ of classes of objects in $\sfA$ such that both $(\sfC, \sfW \cap \sfF)$ and $(\sfC \cap \sfW, \sfF)$ are complete cotorsion pairs (see \ref{cotpair} for the definition of complete cotorsion pairs) and $\sfW$ is thick (i.e., $\sfW$ is closed under direct summands, and for each short exact sequence $0 \to M' \to M \to M'' \to 0$ in $\sfA$, the condition that two of the three objects $M'$, $M$ and $M''$ are in $\sfW$ implies that the third one is also in $\sfW$). Therefore, constructions of abelian model structures on $\sfA$ are equivalent to constructions of Hovey triples, which in general are simpler. Furthermore, by introducing compatible cotorsion pairs, Gillespie proved the following result, which provides a more convenient way to construct Hovey triples; see \cite[Theorem 1.1]{G15}. Recall from \cite{G15} that two complete cotorsion pairs $(\sfC, \bsfF)$ and $(\bsfC, \sfF)$ in $\sfA$ are called {\it compatible} if $\bsfC \subseteq \sfC$ (or equivalently,
$\bsfF \subseteq \sfF$) and $\bsfC \cap \sfF = \sfC \cap \bsfF$, and a cotorsion pair $(\sfC,\sfF)$ is called \textit{hereditary} if $\sfC$ is closed under kernels of epimorphisms and $\sfF$ is closed under cokernels of monomorphisms.

\begin{thm*}[Gillespie]
Let $(\sfC, \bsfF)$ and $(\bsfC, \sfF)$ be complete, hereditary and compatible cotorsion pairs in $\sfA$. Then there is a subcategory $\sfW$ in $\sfA$ such that $(\sfC, \sfW, \sfF)$ forms a Hovey triple. Moreover, $\sfW$ can be described as:
\begin{align*}
\sfW & = \{M \mid \text{there is a s.e.s.}\ 0 \to M \to A \to B \to 0 \text{ with } A \in \bsfF \text{ and } B \in \bsfC \}\\
& = \{ M \mid \text{there is a s.e.s.}\ 0 \to A' \to B' \to M \to 0 \text{ with } A' \in \bsfF \text{ and } B' \in \bsfC \}.
\end{align*}
\end{thm*}

The main aim of this paper is to generalize the above theorem, describing a method to construct model structures (which might not be abelian) on general categories. Note that the key result used to establish Hovey's correspondence is that a pair $(\sfC, \sfF)$ of classes of objects in $\sfA$ is a complete cotorsion pair if and only if the pair $(\Mon(\sfC), \Epi(\sfF))$ forms a weak factorization system in $\sfA$; see Theorem \ref{cot-wfs}. Here,
\begin{align*}
\Mon(\sfC) & = \{\alpha\ |\
                 \alpha\ \text {is a monomorphism}\ \text{with}\
                 \Coker(\alpha) \in \sfC\}, \ \textrm{and} \\
\Epi(\sfF) & = \{\alpha\ |\
                 \alpha\ \text {is an epimorphism}\ \text{with}\
                 \Ker(\alpha) \in \sfF\}.
\end{align*}
It is reasonable to consider weak factorization systems in general categories which take the role of complete cotorsion pairs in abelian categories. Therefore, we need to find appropriate conditions on weak factorization systems in general categories which are analogues of the hereditary and compatible conditions for complete cotorsion pairs in abelian categories.

Recall from Joyal \cite[Definition C.0.20]{Jo08} that a class $\calC$ of morphisms in $\sfE$ satisfies {\it the left cancellation property} if
\begin{center}
$\beta\alpha\in\calC$ and $\beta\in\calC$ $\Rightarrow$ $\alpha\in\calC$.
\end{center}
Dually, a class $\calF$ of morphisms in $\sfE$ satisfies {\it the right cancellation property} if
\begin{center}
$\beta\alpha\in\calF$ and $\alpha\in\calF$ $\Rightarrow$ $\beta\in\calF$.
\end{center}
The following theorem tells us that these cancellation properties on weak factorization systems properly generalize the hereditary condition on cotorsion pairs since they coincide in the abelian situation.

\begin{intthm}\label{thmA}
Let $(\sfC, \sfF)$ be a pair of classes of objects in $\sfA$. Then the following are equivalent.
\begin{eqc}
\item $(\sfC, \sfF)$ is a complete and hereditary cotorsion pair.
\item $(\Mon(\sfC), \Epi(\sfF))$ is a weak factorization system such that $\Mon(\sfC)$ satisfies the left cancellation property and $\Epi(\sfF)$ satisfies the right cancellation property.
\item $(\Mon(\sfC), \Epi(\sfF))$ is a weak factorization system such that $\Mon(\sfC)$ satisfies the left cancellation property or $\Epi(\sfF)$ satisfies the right cancellation property.
\end{eqc}
\end{intthm}

We then give the definition of compatible weak factorization systems; see Definition \ref{compatibledef}.  As asserted by the next theorem, the compatible condition on weak factorization systems indeed generalizes the one for cotorsion pairs.

\begin{intthm}\label{thmB}
Let $(\sfC, \bsfF)$ and $(\bsfC, \sfF)$ be two pairs of classes of objects in $\sfA$. Then the following are equivalent.
\begin{eqc}
\item $(\sfC, \bsfF)$ and $(\bsfC, \sfF)$ are compatible complete cotorsion pairs.

\item $(\Mon(\sfC), \Epi(\bsfF))$ and $(\Mon(\bsfC), \Epi(\sfF))$ are compatible weak factorization systems.
\end{eqc}
\end{intthm}

The following result as a generalization of Gillespie's theorem (\cite[Theorem 1.1]{G15}) provides a method to construct model structures on general categories via compatible weak factorization systems satisfying certain conditions. To extend the application of this theorem, we use a slightly more general definition of model structures without assuming that $\sfE$ is bicomplete (see Definition \ref{new df of model structure}), and say that a model structure is {\it hereditary} if both the class of cofibrations and the class of trivial cofibrations satisfy the left cancellation property.

Recall from van den Berg and Garner \cite{VG12} that a weak factorization system $(\calC, \calF)$ in $\sfE$ satisfies the {\it Frobenius property} if $\sfE$ has pullbacks along morphisms in $\calF$, and the morphisms in $\calC$ are preserved under pullbacks along morphisms in $\calF$. Frobenius's name is invoked here, because there is a connection between the Frobenius property for a weak factorization system and Lawvere's Frobenius condition \cite{Law}; see Clementino, Giuli and Tholen \cite{CGT} for an explanation.

\begin{intthm}\label{thmC}
Let $(\calC, \bcalF)$ and $(\bcalC, \calF)$ be two compatible weak factorization systems in $\sfE$ satisfying the following conditions:
\begin{enumerate}
\item $\sfE$ has pushouts along morphisms in $\calC$ and pullbacks along morphisms in $\calF$;

\item $(\bcalC, \calF)$ satisfies the Frobenius property;

\item both $\calC$ and $\bcalC$ satisfy the left cancellation property.
\end{enumerate}
Then $(\calC, \calW_{\bcalC, \bcalF}, \calF)$ forms a hereditary model structure on $\sfE$, where
\[
\calW_{\bcalC,\bcalF} = \{\alpha\ |\
         \alpha\ \text{can be decomposed as}\ \alpha = \widetilde{f}\widetilde{c}\ \text{with}\
         \widetilde{c} \in \bcalC\ \text{and}\ \widetilde{f} \in \bcalF\}.
\]
\end{intthm}

Since weak factorization systems induced by complete cotorsion pairs automatically satisfy the Frobenius property (see Proposition \ref{Frobenius property}), it follows from Theorems \ref{thmA} and \ref{thmB} that the weak factorization systems $(\Mon(\sfC), \Epi(\bsfF))$ and $(\Mon(\bsfC), \Epi(\sfF))$ induced by complete, hereditary and compatible cotorsion pairs $(\sfC, \bsfF)$ and $(\bsfC, \sfF)$ in $\sfA$ satisfy all conditions specified in Theorem \ref{thmC}. We mention that there are non-abelian examples: in Section \ref{sec:examples}, we consider the classical and constructive Kan-Quillen model structures on the category $\sset$ of simplicial sets and the standard projective model structure on the category $\CH$ of nonnegative chain complexes of modules over a ring $R$, and show that the weak factorization systems associated to these model structures satisfy all conditions specified in Theorem \ref{thmC}.

\begin{rmk*}
We shall mention that there do exist some other classical examples of model structures which do not satisfy Condition (2) or (3) specified in Theorem \ref{thmC}; see Examples \ref{exa1} and \ref{exa2}. We hope to use Theorem \ref{thmC} to construct new model structures on some important categories. We believe that this is a very interesting question, but at this moment, we have not obtained a satisfactory achievement.
\end{rmk*}

\section{Weak factorization systems and cotorsion pairs}
\noindent
In this section, we introduce the compatible condition for weak factorization systems, which generalizes the one for cotorsion pairs, and prove Theorems \ref{thmA} and \ref{thmB} as advertised in the introduction.

\begin{bfhpg}[\bf Lifting property]
Let $l: A \to B$ and $r: C \to D$ be two morphisms in $\sfE$. Recall that $l$ has the {\it left lifting property} with respect to $r$ (or $r$ has the {\it right lifting property} with respect to $l$) if for every pair $f: A \to C$ and $g: B \to D$ of morphisms such that $rf = gl$, there exists a morphism $t: B \to C$ such that $f = tl$ and $g = rt$, that is, the following diagram commutes:
\[
\xymatrix{
A \ar[d]_{l} \ar[r]^{f} & C \ar[d]^{r} \\
  B \ar@{.>}[ur]|-{t} \ar[r]_{g} & D.
}
\]
For a class $\calC$ of morphisms in $\sfE$, denote by $\calC^\Box$ the class of morphisms in $\sfE$ having the right lifting property with respect to all morphisms in $\calC$. The class $^\Box\calC$ is defined dually.
\end{bfhpg}

The following definition of weak factorization systems was given by Bousfield \cite{Bo77}.

\begin{dfn}
A pair $(\calC, \calF)$ of classes of morphisms in $\sfE$ is called a {\it weak factorization system} if $\calC^\Box = \calF$ and ${^\Box\calF} = \calC$, and every morphism $\alpha$ in $\sfE$ can be decomposed as $\alpha = fc$ with $c \in \calC$ and $f \in \calF$.
\end{dfn}

\begin{rmk}\label{rmk1.3}
Let $(\calC, \calF)$ be a weak factorization system in $\sfE$. Then

\begin{prt}
\item the classes $\calC$ and $\calF$ are closed under compositions and retracts, and contain the isomorphisms in $\sfE$;

\item if $\sfE$ has pushouts along morphisms in $\calC$, then $\calC$ is closed under pushouts;

\item if $\sfE$ has pullbacks along morphisms in $\calF$, then $\calF$ is closed under pullbacks;

\item $\calC \cap \calF$ is the class of isomorphisms in $\sfE$.
\end{prt}
For details, see \cite[Propositions D.1.2 and D.1.3]{Jo08}.
\end{rmk}

\begin{bfhpg}[\bf Complete cotorsion pairs] \label{cotpair}
A pair $(\sfC,\sfF)$ of classes of objects in $\sfA$ is called a \emph{cotorsion pair} if $\sfC^{\bot} = \sfF$ and $^{\bot}\sfF =\sfC$, where
\begin{align*}
\sfC^{\bot} & = \{M \in \sfA \mid \textrm{Ext}_{\sfA}^{1}(C, M) = 0 \textrm{ for all objects } C \in \sfC \}, \ \textrm{and} \\
^{\bot}\sfF & = \{M \in \sfA \mid \textrm{Ext}_{\sfA}^{1}(M, D) = 0 \textrm{ for all objects } D \in \sfF\}.
\end{align*}
Following Enochs and Jenda \cite{rha}, a cotorsion pair $(\sfC,\sfF)$ is said to be \emph{complete} if for any object $M$ in $\sfA$, there exist short exact sequences $0 \to D \to C \to M \to 0$ and $0 \to M \to D' \to C' \to 0$ in $\sfA$ with $D, D' \in \sfF$ and $C, C' \in \sfC$.
\end{bfhpg}

Recall that for a class $\sfC$ of objects in $\sfA$,
\begin{align*}
\Mon(\sfC) & = \{\alpha \mid
                 \alpha\ \text {is a monomorphism with } \Coker(\alpha) \in \sfC\}, \textrm{ and}\\
\Epi(\sfC) & = \{\alpha \mid \alpha \text { is an epimorphism with } \Ker(\alpha) \in \sfC\}.
\end{align*}
The next result is essentially due to Hovey \cite{Ho02}; see also Positselski and \v{S}\v{t}ov\'{\i}\v{c}ek \cite[Theorem 2.4]{PS22}.

\begin{thm}\label{cot-wfs}
A pair $(\sfC, \sfF)$ of classes of objects in $\sfA$ is a complete cotorsion pair if and only if $(\Mon(\sfC), \Epi(\sfF))$ is a weak factorization system in $\sfA$.
\end{thm}

In the following we describe a few auxiliary results before proving Theorems \ref{thmA} and \ref{thmB}.

\begin{lem}\label{resolving}
Let $\sfC$ be a class of objects in $\sfA$. Then the following are equivalent.
\begin{eqc}
\item $\sfC$ is closed under kernels of epimorphisms.

\item $\Mon(\sfC)$ satisfies the left cancellation property.
\end{eqc}
\end{lem}

\begin{prf*}
\eqclbl{i}$\implies$\eqclbl{ii}. Let $\alpha: X \to Y$ and $\beta: Y \to Z$ be morphisms in $\sfA$ such that both $\beta\alpha$ and $\beta$ are in $\Mon(\sfC)$.
In particular, $\beta\alpha$ is a monomorphism, so is $\alpha$.
Hence, we obtain a commutative diagram
\begin{equation*}
\xymatrix{
  0 \ar[r] & X \ar@{=}[d]\ar[r]^{\alpha} & Y \ar[d]^{\beta}\ar[r] & \Coker(\alpha) \ar[d]\ar[r] & 0 \\
  0 \ar[r] & X \ar[r]^{\beta\alpha}      & Z \ar[r]  & \Coker(\beta\alpha) \ar[r]  & 0}
\end{equation*}
of short exact sequences, which induces a short exact sequence
\[
0 \to \Coker(\alpha) \to \Coker(\beta\alpha) \to \Coker(\beta) \to 0
\]
by the Snake Lemma as $\beta$ is a monomorphism. Note that $\sfC$ is closed under the kernels of epimorphisms by assumption, and both $\Coker(\beta\alpha)$ and $\Coker(\beta)$ are in $\sfC$. It follows that $\Coker(\alpha)$ is in $\sfC$ as well. Thus, $\alpha$ is in $\Mon(\sfC)$.

\eqclbl{ii}$\implies$\eqclbl{i}. Take a short exact sequence $0 \to X \xra{\beta} C' \to C \to 0$ in $\sfA$ with $C'$ and $C$ in $\sfC$. Let $\alpha$ be the zero morphism from $0$ to $X$. Then $\beta\alpha$ is in $\Mon(\sfC)$. Note that $\beta$ is in $\Mon(\sfC)$ clearly.
It follows that $\alpha$ is in $\Mon(\sfC)$ as well. Hence, $X$ is contained in $\sfC$.
\end{prf*}

The following result can be proved dually.

\begin{lem}\label{coresolving}
Let $\sfF$ be a class of objects in $\sfA$. Then the following are equivalent.
\begin{eqc}
\item $\sfF$ is closed under cokernels of monomorphisms.

\item $\Epi(\sfF)$ satisfies the right cancellation property.
\end{eqc}
\end{lem}

\begin{lem}\label{test hereditary}
Let $(\sfC, \sfF)$ be a pair of classes of objects in $\sfA$ such that $(\Mon(\sfC), \Epi(\sfF))$ is a weak factorization system. Then $\Mon(\sfC)$ satisfies the left cancellation property if and only if $\Epi(\sfF)$ satisfies the right cancellation property.
\end{lem}

\begin{prf*}
We only prove the ``only if" part as the ``if" part is a dual statement.

Note that $\Mon(\sfC)$ satisfies the left cancellation property. By Lemma \ref{resolving}, $\sfC$ is closed under kernels of epimorphisms. Since $(\sfC, \sfF)$ is a complete cotorsion pair by Theorem \ref{cot-wfs}, it follows from Becker \cite[Corollary 1.1.12]{Be14} that $\sfF$ is closed under the cokernels of monomorphisms. Hence, $\Epi(\sfF)$ satisfies the right cancellation property by Lemma \ref{coresolving}.
\end{prf*}

The next result can be proved in a way similar to the proof of Lemma \ref{resolving}.

\begin{lem}\label{extension}
Let $\sfC$ be a class of objects in $\sfA$. Then the following are equivalent.
\begin{eqc}
\item $\sfC$ is closed under extensions.
\item $\Mon(\sfC)$ and $\Epi(\sfC)$ are closed under compositions.
\item $\Mon(\sfC)$ or $\Epi(\sfC)$ are closed under compositions.
\end{eqc}
\end{lem}

\begin{bfhpg}[\bf Proof of Theorem \ref{thmA}]
The equivalences follow immediately from Theorem \ref{cot-wfs} and Lemmas \ref{resolving}, \ref{coresolving}, \ref{test hereditary} and \ref{extension}.
\qed\end{bfhpg}

We then introduce the compatible condition for weak factorization systems.

\begin{dfn}\label{compatibledef}
Two weak factorization systems $(\calC, \bcalF)$ and $(\bcalC, \calF)$ in $\sfE$ are called {\it compatible} if the following conditions hold:
\begin{prt}
\item[(CP1)] $\bcalC\subseteq\calC$ (or equivalently, $\bcalF\subseteq\calF$);

\item[(CP2)] given composable morphisms $\alpha$ and $\beta$ in $\calF$, if two of the three morphisms $\alpha$, $\beta$ and $\beta\alpha$ are in $\bcalF$, then so is the third one;

\item[(CP3)] given $c \in \bcalC$ and $f \in \calF$, if $fc \in \bcalC$, then $f \in \bcalF$.
\end{prt}
\end{dfn}

We mention that the condition (CP3) was called the span property by Sattler in \cite[Definition 2.3]{sa17}.

Let $(\calC, \bcalF)$ and $(\bcalC, \calF)$ be two weak factorization systems in $\sfE$. Define
\[
\calW_{\bcalC,\bcalF} = \{\alpha \mid \alpha \text{ can be decomposed as } \alpha = \widetilde{f}\widetilde{c} \text{ with } \widetilde{c} \in \bcalC \text{ and } \widetilde{f} \in \bcalF\}.
\]
The following result, first proved by Sattler \cite[Lemma 2.1]{sa17}, will be used frequently in the sequel. For the convenience of the reader, we include a detailed proof.

\begin{lem}\label{weclass}
Let $(\calC, \bcalF)$ and $(\bcalC, \calF)$ be two weak factorization systems in $\sfE$ satisfying the condition $\mathrm{(CP1)}$. Then $\bcalC = \calC \cap \calW_{\bcalC,\bcalF}$ and $\bcalF = \calF \cap \calW_{\bcalC,\bcalF}$.
\end{lem}

\begin{prf*}
We only prove the first equality as the second one can be proved similarly.

It is clear that $\bcalC \subseteq \calC \cap \calW_{\bcalC,\bcalF}$,
so it suffices to show the inclusion of the other direction. Take a morphism $\alpha: X\to Y$ in $\calC \cap \calW_{\bcalC,\bcalF}$. For a morphism $f: A \to B$ in $\calF$ and two morphisms $\lambda: X \to A$ and $\mu: Y\to B$ such that $f \lambda = \mu \alpha$, since $\alpha$ is in $\calW_{\bcalC,\bcalF}$, we can find a morphism $\widetilde{c}: X \to C$ in $\bcalC$ and a morphism $\widetilde{f}: C\to Y$ in $\bcalF$ such that $\alpha = \widetilde{f}\widetilde{c}$ as shown in the following commutative diagram:

\[
\xymatrix@R=0.5cm{
X \ar[dd]_{\alpha}\ar[rr]^{\lambda} \ar[dr]_{\widetilde{c}} && A \ar[dd]^{f}\\
                & C \ar[dl]_{\widetilde{f}} \\
Y \ar[rr]_{\mu}                                             && B.}
\]
Since $\widetilde{c}$ has the left lifting property with respect to $f$, one can find a morphism $h: C \to A$ such that the following diagram commutes:
\[
\xymatrix@R=0.5cm{
X \ar[dd]_{\alpha}\ar[rr]^{\lambda} \ar[dr]_{\widetilde{c}} && A \ar[dd]^{f}\\
                & C \ar[dl]_{\widetilde{f}} \ar[ur]^h \\
Y \ar[rr]_{\mu}                                             && B}
\]
in which the left triangle gives rise to the following commutative square:
$$\xymatrix{
  X \ar[d]_{\alpha} \ar[r]^{\widetilde{c}} & C \ar[d]^{\widetilde{f}} \\
  Y \ar@{=}[r]                             & Y.}$$
Since $\alpha$ has the left lifting property with respect to $\widetilde{f}$, there is a morphism $h': Y \to C$ such that $h'\alpha = \widetilde{c}$ and $\widetilde{f} h' = \id_Y$. Thus, $hh'\alpha = h \widetilde{c} = \lambda$ and $fhh' = \mu \widetilde{f} h' = \mu$, so $\alpha$ has the left lifting property with respect to $f$, and hence, belongs to $^\Box\calF = \bcalC$. Consequently, one has $\calC \cap \calW_{\bcalC,\bcalF} \subseteq \bcalC$.
\end{prf*}

\begin{lem}\label{compatible}
Let $(\calC, \bcalF)$ and $(\bcalC, \calF)$ be two compatible weak factorization systems in $\sfE$. Suppose that the composite $g=fh$ is contained in $\calW_{\bcalC,\bcalF}$. Then one has:
\begin{prt}
\item if $\sfE$ has pullbacks along morphisms in $\bcalF$ and $f\in\bcalF$, then $h\in\calW_{\bcalC,\bcalF}$;

\item if $\sfE$ has pushouts along morphisms in $\bcalC$ and $h \in \bcalC$, then $f\in\calW_{\bcalC,\bcalF}$.
\end{prt}
\end{lem}

\begin{prf*}
We only prove the first statement as the second one can be proved dually.

Write $h: X\to Z$ and $f: Z \to Y$. Note that $g = fh$ is contained in $\calW_{\bcalC,\bcalF}$. By definition, there is a morphism $\widetilde{c}: X\to Z'$ in $\bcalC$ and a morphism $\widetilde{f}: Z'\to Y$ in $\bcalF$ such that the following diagram commutes:
\[
\xymatrix{
X \ar[r]^{\widetilde{c}} \ar[d]_h & Z' \ar[d]^{\widetilde{f}} \\
Z \ar[r]_{f}                                                     & Y.}
\]
Now consider the following pullback diagram:
\[
\xymatrix{
P  \ar@{}[rd]|<<{\ulcorner} \ar[d]_{\widetilde{f_2}} \ar[r]^{\widetilde{f_1}} & Z' \ar[d]^{\widetilde{f}} \\
Z \ar[r]_{f}                                                     & Y.}
\]
Since $f$ is in $\bcalF$ by assumption,
one has that $\widetilde{f_1}$ is in $\bcalF$ as well; see Remark \ref{rmk1.3}(c).
Similarly, note that $\widetilde{f}$ is in $\bcalF$,
it follows that $\widetilde{f_2}$ is also in $\bcalF$.
By the universal property of pullbacks, there is a morphism $\alpha: X \to P$ such that $\widetilde{f_1}\alpha = \widetilde{c}$ and $\widetilde{f_2}\alpha = h$. Note that $(\bcalC, \calF)$ is a weak factorization system. There exist a morphism $\widetilde{c}': X \to P'$ in $\bcalC$ and a morphism $f': P' \to P$ in $\calF$ such that $\alpha = f'\widetilde{c}'$. Putting these pieces of information together, we obtain the following commutative diagram:
\[
\xymatrix@R=0.5cm@C=1.5cm{
                                && Z' \ar[dr]^{\widetilde{f}}\\
X \ar[urr]^{\widetilde{c}}\ar[drr]_{h}\ar[r]|-{\widetilde{c}'} & P' \ar[r]|-{f'}
                                 & P  \ar[d]^{\widetilde{f_2}} \ar[u]_{\widetilde{f_1}} & Y.\\
                                && Z  \ar[ur]_{f}}
\]
Since $\widetilde{f_1}f'\widetilde{c}' = \widetilde{c}$ is contained in $\bcalC$ and $\widetilde{f_1}\in\bcalF\subseteq\calF$ by the condition (CP1), one deduces that $\widetilde{f_1}f'$ belongs to $\bcalF$ by the condition (CP3). Hence, $f' \in \bcalF$ by the condition (CP2). Consequently, $h = (\widetilde{f_2} f') \widetilde{c}'$ is contained in $\calW_{\bcalC, \bcalF}$.
\end{prf*}

\begin{bfhpg}[\bf Proof of Theorem \ref{thmB}]
\eqclbl{i}$\implies$\eqclbl{ii}. By Theorem \ref{cot-wfs}, $(\Mon(\sfC), \Epi(\bsfF))$ and $(\Mon(\bsfC), \Epi(\sfF))$ are two weak factorization systems. It remains to show that they are compatible. By \cite[Theorem 1.1]{G15}, there is a class $\sfW$ of objects in $\sfA$ satisfying the 2-out-of-3 property\footnote{Explicitly, if $0 \to M' \to M \to M'' \to 0$ is a short exact sequence in $\sfA$, then the condition that two of the three objects $M'$, $M$ and $M''$ are in $\sfW$ implies that the third one is in $\sfW$.} such that $\bsfC = \sfC \cap \sfW$ and $\bsfF = \sfF \cap \sfW$.

The condition (CP1) holds clearly. Let $\alpha: X \to Y$ and $\beta: Y \to Z$ be two morphisms in $\Epi(\sfF)$. Then $\beta\alpha$ is in $\Epi(\sfF)$ by Lemma \ref{extension}.
Since $\beta\alpha$ is an epimorphism, we obtain the commutative diagram
\[
\xymatrix{
  0 \ar[r] & \Ker(\alpha) \ar[d]\ar[r] & X \ar@{=}[d]\ar[r]^{\alpha} & Y \ar[d]^{\beta}\ar[r] & 0 \\
  0 \ar[r] & \Ker(\beta\alpha) \ar[r]  & X \ar[r]^{\beta\alpha}      & Z \ar[r]
  & 0}
\]
of short exact sequences, which gives rise to another short exact sequence
\[
0 \to \Ker(\alpha) \to \Ker(\beta\alpha) \to \Ker(\beta) \to 0.
\]
Thus, the condition (CP2) holds as $\sfW$ satisfies the 2-out-of-3 property.

To prove the condition (CP3), let $c: A\to B$ be in $\Mon(\bsfC)$ and $f: B\to C$ be in $\Epi(\sfF)$ such that $g = fc$ is in $\Mon(\bsfC)$.
Then $c$ and $g$ are monomorphisms with $\Coker(c)$ and $\Coker(g)$ in $\bsfC$, and $f$ is an epimorphism with $\Ker(f)$ in $\sfF$.
We want to show that $f$ is in $\Epi(\bsfF)$. It suffices to show that
$\Ker (f)$ is in $\bsfF$.
Indeed, note that the commutative diagram
\[
\xymatrix{
  0 \ar[r] & A \ar@{=}[d]\ar[r]^{c} & B \ar[d]^{f}\ar[r] & \Coker(c) \ar[d]\ar[r] & 0 \\
  0 \ar[r] & A \ar[r]^{g}           & C \ar[r]           & \Coker(g) \ar[r]       & 0}
\]
of short exact sequences induces another short exact sequence
\[
0 \to \Ker (f) \to \Coker (c) \to \Coker (g) \to 0.
\]
Since both $\Coker (c)$ and $\Coker (g)$ are in $\sfW$ as $\bsfC \subseteq \sfW$ and $\sfW$ satisfies the 2-out-of-3 property, it follows that $\Ker (f)$ belongs to $\sfW$ as well. Thus, $\Ker (f)$ is in $\bsfF$ as $\bsfF = \sfF \cap \sfW$.

\eqclbl{ii}$\implies$\eqclbl{i}. By Theorem \ref{cot-wfs}, $(\sfC, \bsfF)$ and $(\bsfC, \sfF)$ are two complete cotorsion pairs in $\sfA$. If an object $X$ is contained in $\bsfC$, then the morphism $0 \to X$ is contained in $\Mon(\bsfC)$, and hence in $\Mon(\sfC)$, so $X \in \sfC$, which yields that $\bsfC \subseteq \sfC$. To complete the proof that $(\bsfC, \sfF)$ and $(\sfC, \bsfF)$ are compatible,
it remains to show that $\sfC \cap \bsfF = \bsfC \cap \sfF$.

Now we prove that $\bsfC \cap \sfF \subseteq \sfC \cap \bsfF$. Let $M$ be an object in $\bsfC \cap \sfF$. Then it is clear that $0 \to M$ is in $\Mon(\bsfC)$ and $M\to 0$ is in $\Epi(\sfF)$. Since $(\Mon(\bsfC), \Epi(\sfF))$ and $(\Mon(\sfC), \Epi(\bsfF))$ are two compatible weak factorization systems in $\sfA$, it follows from the condition (CP3) that $M\to 0$ is in $\Epi(\bsfF)$ as $0\to 0$ is in $\Mon(\bsfC)$. Therefore, $M$ is contained in $\bsfF$. But $M$ is also contained in $\sfC$ as $\bsfC \subseteq \sfC$, so it belongs to $\sfC \cap \bsfF$.
Thus, $\bsfC \cap \sfF \subseteq \sfC \cap \bsfF$.

Next, we show that $\sfC \cap \bsfF \subseteq \bsfC \cap \sfF$. Let $N$ be an object in $\sfC \cap \bsfF$. Then $0 \to N$ is contained in $\Mon(\sfC)$ and $N \to 0$ is contained in $\Epi(\bsfF)$. It follows from Lemma \ref{compatible} that $0 \to N$ is in $\calW_{\Mon(\bsfC), \Epi(\bsfF)}$ as $0 \to 0$ belongs to $\calW_{\Mon(\bsfC), \Epi(\bsfF)}$ and $N \to 0$ belongs to $\Epi(\bsfF)$. Consequently, $0 \to N$ is contained in $\Mon(\bsfC)$ as $\calW_{\Mon(\bsfC), \Epi(\bsfF)} \cap \Mon(\sfC) = \Mon(\bsfC)$ by Lemma \ref{weclass}, so $N \in \bsfC$. On the other hand, note that $N \in \sfF$ as
$\bsfF = \sfC^{\perp} \subseteq \bsfC^{\perp} = \sfF$.
It follows that $N \in \bsfC \cap \sfF$.
Therefore, $\bsfC \cap \sfF \subseteq \sfC \cap \bsfF$.
\qed\end{bfhpg}

We end this section with the following result, which shows particularly that the Frobenius property automatically holds for weak factorization systems induced by complete cotorsion pairs, that is, if $(\sfC, \sfF)$ is a complete cotorsion pair in $\sfA$, then the weak factorization system $(\Mon(\sfC), \Epi(\sfF))$ satisfies the Frobenius property.

\begin{prp}\label{Frobenius property}
Let $\sfC$ be a class of objects in $\sfA$. Then one has:
\begin{prt}
\item the morphisms in $\Mon(\sfC)$ are preserved under pullbacks along epimorphisms;
\item the morphisms in $\Epi(\sfC)$ are preserved under pushouts along monomorphisms.
\end{prt}
\end{prp}

\begin{prf*}
We only prove (a); the statement (b) can be proved dually.

Take two morphisms $\alpha: X \to Z$ and $\beta: Y \to Z$ in $\sfA$ with $\alpha$ in $\Mon(\sfC)$ and $\beta$ an epimorphism. We deduce from the pullback diagram
\[
\xymatrix{
  0 \ar[r]    & P \ar@{}[rd]|<<{\ulcorner}\ar[d]\ar[r]^{\alpha'} & Y \ar[d]^{\beta}\ar[r] & \Coker(\alpha) \ar@{=}[d]\ar[r] & 0 \\
  0 \ar[r]    & X \ar[r]^{\alpha}                          & Z \ar[r] &  \Coker(\alpha) \ar[r] & 0
}
\]
that $\alpha'$ is contained in $\Mon(\sfC)$ as well. Hence, the conclusion follows.
\end{prf*}

\section{From weak factorization systems to model structures}\label{sec1}
\noindent
In this section, we describe a method to construct model structures on general categories via two compatible weak factorization systems satisfying certain conditions, and prove Theorem \ref{thmC} as advertised in the introduction, which generalizes a very useful result by Gillespie for abelian model structures.

We begin with the following definition of model structures, which is a slight generalization of the usual one; see Gambino, Henry, Sattler and Szumi{\l}o \cite{GHS22}.

\begin{bfhpg}[\bf Model structures] \label{new df of model structure}
A {\it model structure} on $\sfE$ is a triple $(\calC, \calW, \calF)$ of classes of morphisms in $\sfE$ satisfying:
\begin{prt}
\item $\sfE$ has pushouts along morphisms in $\calC$ and pullbacks along morphisms in $\calF$;
\item $(\calC, \calW \cap \calF)$ and $(\calC \cap \calW, \calF)$ are weak factorization systems;
\item $\calW$ satisfies the 2-out-of-3 property, i.e., for composable morphisms $\alpha$ and $\beta$, if two of the three morphisms $\alpha$, $\beta$ and $\beta\alpha$ are in $\calW$, then so is the third one.
\end{prt}
Morphisms in $\calC$ (resp., $\calW$, $\calF$) are called {\it cofibrations} (resp., {\it weak equivalences}, {\it fibrations}). Morphisms in $\calC \cap \calW$ (resp., $\calF \cap \calW$) are called {\it trivial cofibrations} (resp., {\it trivial fibrations}). Recall that a model structure is said to be {\it hereditary} if both the class of cofibrations and the class of trivial cofibrations satisfy the left cancellation property.
\end{bfhpg}

\begin{rmk}
With respect to this definition, one can show that $\calW$ is closed under retracts; see the proof of \cite[Proposition E.1.3]{Jo08}. One can also show that the model structure is completely determined by two of its three classes of morphisms. When $\sfE$ is finitely complete and cocomplete, the above definition is equivalent to the classical one in the sense of Quillen \cite{Qui67}.
\end{rmk}

Given a model structure $(\calC, \calW, \calF)$ on $\sfE$, the next result shows that $(\calC, \calW \cap \calF)$ and $(\calC \cap \calW, \calF)$ are two compatible weak factorization systems in $\sfE$.

\begin{prp}\label{mcw}
Let $(\calC, \calW, \calF)$ be a model structure on $\sfE$. Then $(\calC, \calW \cap \calF)$ and $(\calC \cap \calW, \calF)$ are compatible weak factorization systems in $\sfE$.
\end{prp}
\begin{prf*}
It is clear that $(\calC, \calW \cap \calF)$ and $(\calC \cap \calW, \calF)$ are weak factorization systems in $\sfE$. Moreover, the above two weak factorization systems are compatible since the condition (CP1) holds obviously, and the conditions (CP2) and (CP3) follow from the 2-out-of-3 property of $\calW$.
\end{prf*}

In the following, we shall prove Theorem \ref{thmC}, a partial converse statement of Proposition \ref{mcw}, that is, under some mild conditions two compatible weak factorization systems in $\sfE$ induce a model structure on $\sfE$. We need some preparation.

\begin{lem}\label{commutative}
Let $(\calC, \bcalF)$ and $(\bcalC, \calF)$ be two
compatible weak factorization systems in $\sfE$.
Suppose that $\sfE$ has pullbacks along morphisms in $\calF$,
$(\bcalC, \calF)$ satisfies the Frobenius property,
and $\calC$ satisfies the left cancellation property.
If $\alpha \in \calW_{\bcalC, \bcalF}$ and $\beta \in \bcalC$,
then $\beta\alpha$ $($if it is defined$)$ is in $\calW_{\bcalC, \bcalF}$.
\end{lem}

\begin{prf*}
Write $\alpha: X \to Y$ and $\beta: Y \to Z$. Since $(\calC, \bcalF)$ is a weak factorization system in $\sfE$, there exist a morphism $c: X \to T$ in $\calC$ and
a morphism $\widetilde{f}: T \to Z$ in $\bcalF$ such that $\beta \alpha = \widetilde{f} c$. It suffices to show that $c$ is contained in $\bcalC$.

Consider the pullback diagram
\[
\xymatrix{
  P \ar@{}[rd]|<<{\ulcorner}\ar[d]_{\widetilde{c}}\ar[r]^{\widetilde{f}'} & Y \ar[d]^{\beta} \\
  T \ar[r]_{\widetilde{f}}                                                & Z.}
\]
Then $\widetilde{f}'$ is contained in $\bcalF$ as $\widetilde{f}$ is so;
see Remark \ref{rmk1.3}(c). Note that $\widetilde{f}$ is also in $\calF$
as $\bcalF \subseteq \calF$.
Since $\beta$ belongs to $\bcalC$ by assumption,
it follows from the Frobenius property of $(\bcalC, \calF)$ that
$\widetilde{c}$ is in $\bcalC$.
By the commutative diagram
\[
\xymatrix@R=0.5cm{
  X \ar[dd]_{c}\ar[rr]^{\alpha}\ar@{.>}[dr]^{\tau}           && Y \ar[dd]^{\beta}\\
                             & P \ar[dl]_{\widetilde{c}}\ar[ur]_{\widetilde{f}'} \\
  T \ar[rr]_{\widetilde{f}}                                  && Z}
\]
as well as the universal property of pullbacks, there is a morphism $\tau: X \to P$ such that $\widetilde{f}'\tau = \alpha$ and $\widetilde{c}\tau = c$.
Hence, to show that $c$ is in $\bcalC$, it suffices to show that $\tau$ is contained in $\bcalC$. Indeed, it is clear that $\tau$ belongs to $\calC$ as $\calC$ satisfies the left cancellation property. Moreover, since $\widetilde{f}'\tau = \alpha \in \calW_{\bcalC, \bcalF}$, one has $\tau \in \calW_{\bcalC, \bcalF}$ by Lemma \ref{compatible}. Thus, $\tau \in \calW_{\bcalC, \bcalF} \cap \calC = \bcalC$ by Lemma \ref{weclass}, as desired.
\end{prf*}

To prove that $(\calC, \calW_{\bcalC, \bcalF}, \calF)$ is a model structure on $\sfE$, by Lemma \ref{weclass}, it suffices to show that $\calW_{\bcalC, \bcalF}$ satisfies the 2-out-of-3 property, which is established in the following three lemmas.

\begin{lem}\label{composition}
Under the same assumptions as specified in Lemma \ref{commutative}, if $\alpha \in \calW_{\bcalC, \bcalF}$ and $\beta \in \calW_{\bcalC, \bcalF}$, then $\beta\alpha$
$($if it is defined$)$ is in $\calW_{\bcalC, \bcalF}$ as well.
\end{lem}

\begin{prf*}
Write $\alpha: X \to Y$ and $\beta: Y \to Z$. Then we can find morphisms $\widetilde{c}: X \to X'$ and $\widetilde{c}': Y \to Y'$ in $\bcalC$ together with morphisms $\widetilde{f}: X' \to Y$ and $\widetilde{f}': Y' \to Z$ in $\bcalF$ such that $\alpha = \widetilde{f}\widetilde{c}$ and $\beta = \widetilde{f}' \widetilde{c}'$; see the commutative diagram
\[
\xymatrix@R=0.5cm{
X \ar[rr]^{\alpha} \ar[dr]_{\widetilde{c}} && Y \ar[rr]^{\beta} \ar[dr]_{\widetilde{c}'} && Z.\\
                & X' \ar[ur]_{\widetilde{f}}  && Y'  \ar[ur]_{\widetilde{f}'}}
\]
By Lemma \ref{commutative}, one has $\widetilde{c}'\widetilde{f} \in \calW_{\bcalC,\bcalF}$, so there are morphisms $\widetilde{c}'': X' \to T$ in $\bcalC$ and  $\widetilde{f}'': T \to Y'$ in $\bcalF$ satisfying $\widetilde{c}' \widetilde{f} = \widetilde{f}'' \widetilde{c}''$; see the commutative diagram
\[
\xymatrix@R=0.5cm{
X \ar[rr]^{\alpha} \ar[dr]_{\widetilde{c}} && Y \ar[rr]^{\beta} \ar[dr]_{\widetilde{c}'} && Z.\\
                & X' \ar[ur]_{\widetilde{f}}\ar[dr]_{\widetilde{c}''}  && Y'  \ar[ur]_{\widetilde{f}'}\\
                &&T \ar[ur]_{\widetilde{f}''}}\\
\]
Consequently, $\beta \alpha = \widetilde{f}' \widetilde{f}'' \widetilde{c}'' \widetilde{c}$, which is contained in $\calW_{\bcalC, \bcalF}$.
\end{prf*}

\begin{lem}\label{(2-3)property1}
Under the same assumptions as specified in Lemma \ref{commutative} and the extra condition that $\sfE$ has pushouts along morphisms in $\bcalC$, if $\alpha \in \calW_{\bcalC, \bcalF}$ and $\beta\alpha \in \calW_{\bcalC, \bcalF}$,
then $\beta \in \calW_{\bcalC, \bcalF}$.
\end{lem}

\begin{prf*}
Write $\alpha: X \to Y$ and $\beta: Y \to Z$. Since $(\bcalC, \calF)$ is a weak factorization system in $\sfE$, there is a decomposition $\beta = f'\widetilde{c}'$ with $\widetilde{c}': Y \to Y'$ in $\bcalC$ and $f': Y' \to Z$ in $\calF$. It suffices to show that $f'$ is contained in $\bcalF$.

Since $\alpha \in \calW_{\bcalC, \bcalF}$, one gets a decomposition $\alpha = \widetilde{f}\widetilde{c}$ with $\widetilde{c}: X \to X'$ in $\bcalC$ and $\widetilde{f}: X' \to Y$ in $\bcalF$. By Lemmas \ref{weclass} and \ref{composition}, $\widetilde{c}'\widetilde{f}$ is in $\calW_{\bcalC, \bcalF}$, so it also has a decomposition $\widetilde{c}'\widetilde{f} = \widetilde{f}''\widetilde{c}''$ with $\widetilde{c}'': X' \to T$ in $\bcalC$ and  $\widetilde{f}'': T \to Y'$ in $\bcalF$; see the commutative diagram
\[
\xymatrix@R=0.5cm{
X \ar[rr]^{\alpha} \ar[dr]_{\widetilde{c}} && Y \ar[rr]^{\beta} \ar[dr]_{\widetilde{c}'} && Z.\\
                & X' \ar[ur]_{\widetilde{f}}\ar[dr]_{\widetilde{c}''}  && Y'  \ar[ur]_{f'}\\
                &&T \ar[ur]_{\widetilde{f}''}}\\
\]
Consequently, one has
\[
\beta \alpha = f'\widetilde{c}' \widetilde{f} \widetilde{c} = f'\widetilde{f}''\widetilde{c}''\widetilde{c}.
\]
Note that $\beta\alpha \in \calW_{\bcalC, \bcalF}$ and $\widetilde{c}''\widetilde{c} \in \bcalC$, and $\sfE$ has pushouts along morphisms in $\bcalC$. It follows from Lemma \ref{compatible}(b) that $f'\widetilde{f}''$ is in $\calW_{\bcalC, \bcalF}$, and so one has $f'\widetilde{f}'' \in \bcalF$ by Lemma \ref{weclass}.
The condition (CP2) tells us that $f' \in \bcalF$, as desired.
\end{prf*}

\begin{lem}\label{(2-3)property2}
Under the same assumptions as specified in Lemma \ref{commutative} and the extra conditions that $\sfE$ has pushouts along morphisms in $\bcalC$ and $\bcalC$ satisfies the left cancellation property, if $\beta \in \calW_{\bcalC, \bcalF}$ and $\beta\alpha \in \calW_{\bcalC, \bcalF}$, then $\alpha \in \calW_{\bcalC, \bcalF}$.
\end{lem}

\begin{prf*}
The proof is very similar to that of Lemma \ref{(2-3)property1}, so we only give a sketch to illustrate the application of the additional condition that $\bcalC$ satisfies the left cancellation property.

Write $\alpha: X\to Y$ and $\beta: Y\to Z$. Decompose $\alpha = \widetilde{f}c$ with $c: X \to X'$ in $\calC$ and $\widetilde{f}: X' \to Y$ in $\bcalF$. It remains to show that $c \in \bcalC$. As we did in the proof of Lemma \ref{(2-3)property1}, there is a decomposition $\beta\alpha = \widetilde{f}' \widetilde{f}'' \widetilde{c}'' c$. Since $\beta\alpha \in \calW_{\bcalC, \bcalF}$ and $\widetilde{f}'\widetilde{f}'' \in \bcalF$, one has $\widetilde{c}''c \in \calW_{\bcalC, \bcalF}$ by Lemma \ref{compatible}.
Hence, by Lemma \ref{weclass}, one concludes that $\widetilde{c}''c \in \bcalC$. Now the left cancellation property of $\bcalC$ tells us that $c \in \bcalC$.
\end{prf*}

We can now give the proof of Theorem \ref{thmC}.

\begin{bfhpg}[\bf Proof of Theorem \ref{thmC}]
By Lemma \ref{weclass}, we see that the pairs $(\calC, \calW_{\bcalC, \bcalF} \cap \calF)$ and
$(\calC \cap \calW_{\bcalC, \bcalF}, \calF)$ are weak factorization systems in $\sfE$. Moroever, $\calW_{\bcalC, \bcalF}$ satisfies the 2-out-of-3 property by Lemmas \ref{composition}, \ref{(2-3)property1} and \ref{(2-3)property2}. Thus, $(\calC, \calW_{\bcalC, \bcalF}, \calF)$ forms a model structure on $\sfE$, and furthermore, it is hereditary by the condition (3).
\qed\end{bfhpg}

\begin{rmk}
In Theorem \ref{thmC}, the class $\calW_{\bcalC, \bcalF}$ in the model structure $(\calC, \calW_{\bcalC, \bcalF}, \calF)$ is unique. Indeed, if $(\calC, \calW, \calF)$ is another model structure on $\sfE$, then $\calC \cap \calW = \calC \cap \calW_{\bcalC, \bcalF}$ and $\calW \cap \calF = \calW_{\bcalC, \bcalF} \cap \calF$. It is easy to check that $\calW = \calW_{\bcalC, \bcalF}$.
\end{rmk}

The following two examples show that there indeed exist model structures which do not satisfy the condition (2) or (3) in Theorem \ref{thmC}. Recall that a model structure is {\it right proper} if weak equivalences are preserved under pullbacks along fibrations. We mention that a model structure $(\calC, \calW, \calF)$ on $\sfE$ in which the cofibrations are preserved under pullbacks is right proper if and only if the weak factorization system
$(\calC\cap\calW, \calF)$ satisfies the Frobenius property; see \cite{VG12}.

\begin{exa}\label{exa1}
Let $\sset$ denote the category of simplicial sets, defined as usual to be the category of presheaves of sets over the simplex category $\triangle$. There is a model structure $(\calC, \calW, \calF)$ on $\sset$ in which weak equivalences are rational homology isomorphisms and cofibrations are inclusions; see Quillen \cite{Quillen69}. Thus the cofibrations are preserved under pullbacks. Consider the corresponding compatible weak factorization systems $(\calC, \bcalF)$ and $(\bcalC, \calF)$ with $\bcalF=\calW\cap\calF$ and $\bcalC=\calC\cap\calW$; see Proposition \ref{mcw}. It is clear that both $\calC$ and $\bcalC$ satisfy the left cancellation property.
However, this model structure is not right proper; see \cite[p.p. 71]{Rezk02}.
Thus, the weak factorization system $(\bcalC, \calF)$ does not satisfy the Frobenius property.
\end{exa}

\begin{exa}\label{exa2}
Let $\Set$ denote the category of sets. It is well known that there is a model structure $(\calC, \calW, \calF)$ on $\Set$, where $\calC$ is the class of surjective maps, $\calF$ is the class of injective maps, and $\calW$ is the class of all maps. Clearly, in this model structure, both $\calC$ and $\bcalC=\calC\cap\calW$ are the class of surjective maps, and hence, do not satisfy the left cancellation property. However, the weak factorization system $(\bcalC, \calF)$ does satisfy the Frobenius property since $\Set$ is a regular category in which every epimorphism is a regular epimorphism.
\end{exa}

\section{Examples of compatible weak factorization systems}\label{sec:examples}
\noindent
It follows from Theorems \ref{thmA} and \ref{thmB} and Proposition \ref{Frobenius property} that if $(\sfC, \bsfF)$ and $(\bsfC, \sfF)$ are compatible complete hereditary cotorsion pairs in an abelian category $\sfA$, then the induced weak factorization systems $(\Mon(\sfC), \Epi(\bsfF))$ and $(\Mon(\bsfC), \Epi(\sfF))$ are compatible satisfying all conditions specified in Theorem \ref{thmC}. In this section, we list some examples of compatible weak factorization systems associated to non-abelian model structures: the classical and constructive Kan-Quillen model structures on the category $\sset$ of simplicial sets and the standard projective model structure on the category $\CH$ of nonnegative chain complexes of modules over a ring $R$, and show that the weak factorization systems associated to these model structures satisfy all conditions specified in Theorem \ref{thmC}.

\subsection{Classical Kan-Quillen model structure}
Let $(\calC, \calW, \calF)$ be the Kan-Quillen model structure on $\sset$;
see \cite{Qui67}. Explicitly,
\begin{itemize}
\item \emph{cofibrations} (morphisms in $\calC$) are monomorphisms, i.e., morphisms $f: X \to Y$ in $\sset$ such that $f_k: X_k \to Y_k$ is an injection of sets for each $k\in\mathbb{N}$;
\item \emph{weak equivalences} (morphisms in $\calW$) are morphisms $f: X \to Y$ in $\sset$ whose geometric realization $|f|$ is a weak homotopy equivalence of topological spaces;
\item \emph{fibrations} (morphisms in $\calF$) are the Kan fibrations, i.e., morphisms in $\sset$ that have the right lifting property with respect to all horn inclusions.
\end{itemize}

In the following, let $\bcalF=\calF\cap\calW$ and $\bcalC=\calC\cap\calW$.

\begin{prp}\label{simplicial sets}
The weak factorization systems $(\calC, \bcalF)$ and $(\bcalC, \calF)$ satisfy all conditions specified in Theorem \ref{thmC}.
\end{prp}
\begin{prf*}
The compatible condition is clear since the condition (CP1) holds obviously, and the conditions (CP2) and (CP3) follow from the 2-out-of-3 property of $\calW$. The condition (1) in Theorem \ref{thmC} automatically holds since $\sset$ is bicomplete. It is clear that $\calC$ and $\bcalC$ satisfy the left cancellation property. Finally, it follows from Gambino and Sattler \cite[Theorem 4.8]{GS17} that the weak factorization system $(\bcalC, \calF)$ satisfies the Frobenius property.
\end{prf*}

\subsection{Constructive Kan-Quillen model structure}
The original proofs of the existence of the classical Kan-Quillen model structure on $\sset$ use the law of excluded middle (EM) and the axiom of choice (AC), which are not valid in constructive mathematics. Recently, a constructively valid model structure on $\sset$ was given by Henry \cite{Hen22} and Gambino, Sattler and Szumi\l{o} \cite{GSS22}, which coincides with the classical Kan-Quillen model structure once (EM) and (AC) are assumed.

For the convenience of the reader, we include some details on the constructive Kan-Quillen model structure on $\sset$; for more details, please refer to \cite{GSS22}. A map $i: A \to B$ of sets is called a {\it decidable inclusion} if there is a map $j: C \to B$ such that $i$ and $j$ exhibit $B$ as a coproduct of $A$ and $C$, that is, the diagram
\[
\xymatrix{
  \emptyset \ar[d]_{}\ar[r]^{} & C \ar[d]^{j} \\
  A \ar[r]_{i}                                                & B}
\]
is a pushout and $B \cong A \sqcup C$. The following lemma asserts that the class of decidable inclusions satisfies the left cancellation property, which is trivially holds true while assuming (EM) as in this case decidable inclusions are precisely injections.

\begin{lem} \label{LCP}
Let $\alpha: A \to B$ and $\beta: B \to C$ be two maps of sets, and suppose that both $\beta$ and $\beta \alpha$ are decidable inclusions. Then $\alpha$ is also a decidable inclusion.
\end{lem}

\begin{proof}
By definition, we have the following two pushout diagrams,
which are also pullback diagrams by \cite[Lemma 2.1.1]{GSS22}:
\[
\xymatrix{
\emptyset \ar[r] \ar[d] & A' \ar[d]\\
A \ar[r]^-{\beta \alpha} & C \cong A \sqcup A'
}
\quad
\xymatrix{
\emptyset \ar[r] \ar[d] & B' \ar[d]^-{\beta'}\\
B \ar[r]^-{\beta} & C \cong B \sqcup B'.
}
\]
Note that the map $\beta': B' \to C$ is also a decidable inclusion.

Now consider the intersection $A \cap B'$, which is given by the pullback diagram
\[
\xymatrix{
A \cap B' \ar[r] \ar[d] & B' \ar[d]\\
A \ar[r]^-{\beta \alpha} & C.
}
\]
We then have a commutative diagram
\[
\xymatrix{
 & A \cap B' \ar[dl] \ar@{-->}[d] \ar[dr]\\
 A \ar[dr]^-{\alpha} & \emptyset \ar[r] \ar[d] & B' \ar[d]^-{\beta'}\\
 & B \ar[r]^{\beta} & C,
}
\]
where the right square is a pullback diagram.
Consequently, $A \cap B' = \emptyset$ as there is a map from the intersection to the empty set. Thus, the union $A \cup B'$, given by the following pushout diagram, is actually a coproduct.
\[
\xymatrix{
A \cap B' = \emptyset \ar[r] \ar[d] & B' \ar[d]\\
A \ar[r] & A \cup B' \cong A \sqcup B'
}
\]
By \cite[Lemma 2.1.7]{GSS22}, the map $A \sqcup B' \to C$ is again a decidable inclusion. Consider the following diagram
\[
\xymatrix{
A \ar[r] \ar[d]_-{\alpha} & A \sqcup B' \ar[d]\\
B \ar[r] & C \cong B \sqcup B',
}
\]
which is clearly a pullback diagram. Therefore, it follows from \cite[Lemma 2.1.4]{GSS22} that $\alpha$ is a decidable inclusion.
\end{proof}

Let $I$ (resp., $J$) be the class of boundary inclusions (resp., horn inclusions) in $\sset$. A morphism in $\sset$ is a {\it trivial fibration} (resp., {\it Kan fibration}) if it has the right lifting property with respect to morphisms in $I$ (resp., $J$). A morphism in $\sset$ is a {\it trivial cofibration} (resp., {\it cofibration}) if it has the left lifting property with respect to Kan fibrations (resp., trivial fibrations). A simplicial set $X$ is called {\it cofibrant} if the morphism $\emptyset\to X$ is a cofibration. Let $\sset_{\mathrm{cof}}$ denote the full subcategory of $\sset$ consisting of cofibrant simplicial sets, and set
\begin{itemize}
\item $\calC_{\mathrm{cof}}$: the class of morphisms in $\sset_{\mathrm{cof}}$ that are cofibrations;
\item $\bcalC_{\mathrm{cof}}$: the class of morphisms in $\sset_{\mathrm{cof}}$ that are trivial cofibrations;
\item $\calF_{\mathrm{cof}}$: the class of morphisms in $\sset_{\mathrm{cof}}$ that are Kan fibrations;
\item $\bcalF_{\mathrm{cof}}$: the class of morphisms in $\sset_{\mathrm{cof}}$ that are trivial fibrations;
\item $\calW_{\mathrm{cof}}$: the class of morphisms in $\sset_{\mathrm{cof}}$ that are weak homotopy equivalences in the sense of \cite[3.1]{GSS22}.
\end{itemize}

It follows from \cite[Proposition 2.2.7]{GSS22} that the pairs $(\calC_{\mathrm{cof}}, \bcalF_{\mathrm{cof}})$ and $(\bcalC_{\mathrm{cof}}, \calF_{\mathrm{cof}})$ are weak factorization systems in $\sset_{\mathrm{cof}}$, and furthermore, $\bcalF_{\mathrm{cof}} = \calF_{\mathrm{cof}} \cap \calW_{\mathrm{cof}}$ and $\bcalC_{\mathrm{cof}} = \calC_{\mathrm{cof}} \cap \calW_{\mathrm{cof}}$ by \cite[Propositions 3.6.3 and 3.6.4]{GSS22}. We have the next result.

\begin{prp}\label{cons simplicial sets}
The weak factorization systems $(\calC_{\mathrm{cof}}, \bcalF_{\mathrm{cof}})$ and $(\bcalC_{\mathrm{cof}}, \calF_{\mathrm{cof}})$ satisfy all conditions specified in Theorem \ref{thmC}.
\end{prp}
\begin{prf*}
The compatible condition is clear since the condition (CP1) holds obviously, and the conditions (CP2) and (CP3) follow from the 2-out-of-3 property of $\calW_{\mathrm{cof}}$; see \cite[Lemma 3.1.5]{GSS22}. The condition (1) in Theorem \ref{thmC} automatically holds.

Now we check the left cancellation property. Note that if $\calC_{\mathrm{cof}}$ has this property, then so does $\bcalC_{\mathrm{cof}}$. Therefore, we only need to show that $\calC_{\mathrm{cof}}$ has left cancellation property. Let $f: X \to Y$ and $g: Y \to Z$ be two morphisms in $\sset_{\mathrm{cof}}$ such that both $g$ and $gf$ are cofibrations. By \cite[Corollary 2.4.6]{GSS22}, for each $k \in \mathbb{N}$, both $g_k: Y_k \to Z_k$ and $(gf)_k = g_kf_k: X_k \to Z_k$ are decidable inclusions, so it follows from Lemma \ref{LCP} that each $f_k: X_k \to Y_k$ is also a decidable inclusion. Applying \cite[Corollary 2.4.6]{GSS22} again, we conclude that $f: X \to Y$ is a cofibration.

Finally, it follows from  \cite[Proposition 4.1.6]{GSS22} that the weak factorization system $(\bcalC_{\mathrm{cof}}, \calF_{\mathrm{cof}})$ satisfies the Frobenius property.
\end{prf*}

\subsection{Projective model structure}
Throughout this subsection, all $R$-modules are left $R$-modules. An object $X =\{X_k\}_{k \geqslant 0}$ in $\CH$ is called \emph{acyclic} if the homology groups $\h_k(X)$ vanish for all $k \geqslant 0$. Let $(\calC, \calW, \calF)$ be the standard projective model structure on $\CH$;
see Dwyer and Spali\'{n}ski \cite{DS95model}. Explicitly,
\begin{itemize}
\item \emph{cofibrations} (morphisms in $\calC$) are morphisms $f : X \to Y$ such that $f_k : X_k \to Y_k$ is a monomorphism whose cokernel is a projective $R$-module for each $k \geqslant 0$;

\item \emph{weak equivalences} (morphisms in $\calW$) are morphisms $f: X \to Y$ which induces isomorphisms $\h_k(X) \cong \h_k(Y)$ for all $k \geqslant 0$;

\item \emph{fibrations} (morphisms in $\calF$) are morphisms $f: X \to Y$ such that $f_k : X_k \to Y_k$ is an epimorphism for each $k > 0$.
\end{itemize}
Although $\CH$ is a bicomplete abelian category, this model structure is not abelian since fibrations are required to be epic only in positive degrees.

In the following, let $\bcalF=\calF\cap\calW$ and $\bcalC=\calC\cap\calW$.

\begin{prp}\label{establish wfs}
The weak factorization systems $(\calC, \bcalF)$ and $(\bcalC, \calF)$ satisfy all conditions specified in Theorem \ref{thmC}.
\end{prp}

\begin{prf*}
The compatible condition is clear since the condition (CP1) holds obviously, and the conditions (CP2) and (CP3) follow from  the 2-out-of-3 property of $\calW$. The condition (1) in Theorem \ref{thmC} automatically holds since $\CH$ is bicomplete.

Now we check the left cancellation property. Note that if $\calC$ has this property, then so does $\bcalC$. Therefore, we only need to show that $\calC$ has left cancellation property. Given morphisms $f: X\to Y$ and $g: Y\to Z$ such that $gf \in \calC$ and $g \in \calC$, we want to show that $f$ belongs to $\calC$ as well, that is, for each $k \geqslant 0$, $f_k : X_k \to Y_k$ is a monomorphism (which holds trivially since $g_kf_k = (gf)_k$ is a monomorphism) such that its cokernel is a projective $R$-module. This is also clear. Indeed, applying the Snake Lemma to the commutative diagram
\[
\xymatrix{
0 \ar[r] & X_k \ar@{=}[r] \ar[d]^-{f_k} & X_k \ar[r] \ar[d]^-{g_kf_k} & 0 \ar[r] \ar[d] & 0 \\
0 \ar[r] & Y_k \ar[r]^-{g_k} & Z_k \ar[r] & \Coker(g_k) \ar[r] & 0
}
\]
of short exact sequences, we obtain a short exact sequence
\[
0 \to \Coker(f_k) \to \Coker(g_kf_k) \to \Coker(g_k) \to 0
\]
of $R$-modules. Since both $\Coker(g_kf_k)$ and $\Coker(g_k)$ are projective, so is $\Coker(f_k)$.

Finally, we verify the Frobenius property of $(\bcalC, \calF)$. Given a pullback diagram
\[
\xymatrix{
M \ar@{}[rd]|<<{\ulcorner} \ar[d]_{} \ar[r]^{p} & N \ar[d]^{g} \\
X \ar[r]^{f}                                     & Y}
\]
in $\CH$ with $f \in \bcalC$ and $g \in \calF$. We want to show that $p$ belongs to $\bcalC$. Since $p$ is monic as $f$ is so, it remains to show that $\Coker (p)$ is acyclic and each $\Coker (p_k)$ is a projective $R$-module for $k \geqslant 0$. We do this by proving that the morphism $r$ in the commutative diagram
\[
\xymatrix{
0 \ar[r] & M \ar[r]^-{p} \ar[d]_-{} & N \ar[r] \ar[d]^-{g} & \Coker (p) \ar[r] \ar[d]^-{r} & 0 \\
0 \ar[r] & X \ar[r]^-{f}             & Y \ar[r]            & \Coker (f) \ar[r] & 0}
\]
of short exact sequences is an isomorphism. Indeed, in this case $\Coker (p) \cong \Coker(f)$ and $\Coker(f)$ has the desired property since $f$ is contained in $\bcalC$.

For $k \geqslant 1$, the left square in the commutative diagram
\[
\xymatrix{
0 \ar[r] & M_k \ar[r]^-{p_k} \ar[d]_-{} & N_k \ar[r] \ar[d]^-{g_k} & \Coker(p_k) \ar[r] \ar[d]^-{r_k} & 0 \\
0 \ar[r] & X_k \ar[r]^-{f_k}             & Y_k \ar[r]            & \Coker(f_k) \ar[r] & 0}
\]
of short exact sequences is both a pullback diagram and a pushout diagram since $g_k$ is an epimorphism. Thus, $r_k$ is an isomorphism for $k \geqslant 1$. When $k = 0$, consider the following commutative diagram of $R$-modules with exact rows:
\[
\xymatrix@R=0.7cm@C=0.5cm{
0 \ar[rr] && M_1 \ar[rr]^(0.33){\ \ \ \ \ p_1}\ar[dr] \ar'[d]_-{}[dd]&& N_1 \ar[rr]\ar[dr] \ar'[d][dd]
&& \Coker(p_1)  \ar[rr]\ar'[d][dd] \ar[dr] && 0 \\
& 0 \ar[rr] && M_0 \ar[rr]^(0.33){p_0} \ar[dd]&& N_0 \ar[rr]\ar[dd] && \Coker(p_0)  \ar[rr] \ar[dd]_(0.33){r_0} && 0 \\
0 \ar[rr] && X_1 \ar'[r]^(0.66){f_1}[rr] \ar[dr]&& Y_1 \ar'[r]^(0.5){}[rr] \ar[dr]&& \Coker(f_1)  \ar[rr] \ar[dr]&& 0 \\
& 0 \ar[rr] && X_0 \ar[rr]^(0.33){\ \ \ \ \ f_0} && Y_0 \ar[rr] && \Coker(f_0) \ar[rr] && 0.
}\]
Note that $\Coker (f)$ is acyclic. The morphism $\Coker(f_1) \to \Coker(f_0)$ is an epimorphism, so $r_0$ is also epic as $r_1$ is an isomorphism. On the other hand, since
\[
\xymatrix{
 M_0 \ar[r]^-{p_0} \ar[d]_-{} & N_0  \ar[d]^-{g_0}   \\
 X_0 \ar[r]^-{f_0}               & Y_0             }
\]
is a pullback diagram, it follows that $r_0$ is a monomorphism as well, and hence, an isomorphism. Thus, $r$ is indeed an isomorphism, and our proof is complete.
\end{prf*}

\section*{Acknowledgment}
\noindent
We thank James Gillespie for kindly answering our questions and providing many valuable suggestions, and we extend our gratitude to the referee for pertinent comments that improved the presentation at several points and led to Examples \ref{exa1} and \ref{exa2}.

\bibliographystyle{amsplain-nodash}


\begin{thebibliography}{10}

\bibitem{Be14}
Hanno Becker, \emph{Models for singularity categories}, Adv. Math. \textbf{254}
  (2014), 187--232. \MR{MR3161097}

\bibitem{Bo77}
Aldridge~K. Bousfield, \emph{Constructions of factorization systems in
  categories}, J. Pure Appl. Algebra \textbf{9} (1976/77), no.~2, 207--220.
  \MR{MR478159}

\bibitem{CGT}
Maria~M. Clementino, Eraldo Giuli, and Walter Tholen, \emph{Topology in a
  category: compactness}, Portugal. Math. \textbf{53} (1996), no.~4, 397--433.
  \MR{MR1432147}

\bibitem{DS95model}
William~G. Dwyer and Jan Spali\'{n}ski, \emph{Homotopy theories and model
  categories}, Handbook of algebraic topology, North-Holland, Amsterdam, 1995,
  pp.~73--126. \MR{MR1361887}

\bibitem{rha}
Edgar~E. Enochs and Overtoun M.~G. Jenda, \emph{Relative homological algebra},
  de Gruyter Expositions in Mathematics, vol.~30, Walter de Gruyter \& Co.,
  Berlin, 2000. \MR{MR1753146}

\bibitem{GHS22}
Nicola Gambino, Simon Henry, Christian Sattler, and Karol Szumi{\l}o, \emph{The
  effective model structure and {$\infty $}-groupoid objects}, Forum Math.
  Sigma \textbf{10} (2022), Paper No. e34, 1--59. \MR{MR4436592}

\bibitem{GS17}
Nicola Gambino and Christian Sattler, \emph{The {F}robenius condition, right
  properness, and uniform fibrations}, J. Pure Appl. Algebra \textbf{221}
  (2017), no.~12, 3027--3068. \MR{3666736}

\bibitem{GSS22}
Nicola Gambino, Christian Sattler, and Karol Szumi\l{o}, \emph{The constructive
  {K}an-{Q}uillen model structure: two new proofs}, Q. J. Math. \textbf{73}
  (2022), no.~4, 1307--1373. \MR{MR4520221}

\bibitem{G15}
James Gillespie, \emph{How to construct a {H}ovey triple from two cotorsion
  pairs}, Fund. Math. \textbf{230} (2015), no.~3, 281--289. \MR{MR3351474}

\bibitem{Hen22}
Simon Henry, \emph{A constructive account of the Kan-Quillen model structure
  and of Kan's $\mathrm{Ex}^{\infty}$ functor}, preprint
  \textbf{\arxiv[CT]{1905.06160}}.

\bibitem{Ho02}
Mark Hovey, \emph{Cotorsion pairs, model category structures, and
  representation theory}, Math. Z. \textbf{241} (2002), no.~3, 553--592.
  \MR{MR1938704}

\bibitem{Jo08}
Andr\'{e} Joyal, \emph{The theory of quasi-categories and its applications},
  Quaderns, vol.~45, Centre de Recerca Matem\`{a}tica, 2008.

\bibitem{Law}
F.~William Lawvere, \emph{Equality in hyperdoctrines and comprehension schema
  as an adjoint functor}, Applications of {C}ategorical {A}lgebra ({P}roc.
  {S}ympos. {P}ure {M}ath., {V}ol. {XVII}, {N}ew {Y}ork, 1968), Proc. Sympos.
  Pure Math., XVII, Amer. Math. Soc., Providence, RI, 1970, pp.~1--14.
  \MR{MR257175}

\bibitem{PS22}
Leonid Positselski and Jan \v{S}\v{t}ov\'{\i}\v{c}ek, \emph{Derived, coderived,
  and contraderived categories of locally presentable abelian categories}, J.
  Pure Appl. Algebra \textbf{226} (2022), no.~4, Paper No. 106883, 39.
  \MR{MR4310048}

\bibitem{Qui67}
Daniel~G. Quillen, \emph{Homotopical algebra}, Lecture Notes in Mathematics,
  No. 43, Springer-Verlag, Berlin-New York, 1967. \MR{MR0223432}

\bibitem{Quillen69}
Daniel~G. Quillen, \emph{Rational homotopy theory}, Ann. of Math. (2) \textbf{90} (1969), 205--295. \MR{MR258031}

\bibitem{Rezk02}
Charles Rezk, \emph{Every homotopy theory of simplicial algebras admits a proper
model}, Topology Appl. \textbf{119} (2002), 65--94. \MR{MR1881711}

\bibitem{sa17}
Christian Sattler, \emph{The equivalence extension property and model
  structures}, preprint \textbf{\arxiv[CT]{1704.06911}}.

\bibitem{VG12}
Benno van~den Berg and Richard Garner, \emph{Topological and simplicial models
  of identity types}, ACM Trans. Comput. Log. \textbf{13} (2012), no.~1, Art.
  3, 44. \MR{MR2893018}

\end{thebibliography}

\def\cprime{$'$}
  \providecommand{\arxiv}[2][AC]{\mbox{\href{http://arxiv.org/abs/#2}{\sf
  arXiv:#2 [math.#1]}}}
  \providecommand{\oldarxiv}[2][AC]{\mbox{\href{http://arxiv.org/abs/math/#2}{\sf
  arXiv:math/#2
  [math.#1]}}}\providecommand{\MR}[1]{\mbox{\href{http://www.ams.org/mathscinet-getitem?mr=#1}{#1}}}
  \renewcommand{\MR}[1]{\mbox{\href{http://www.ams.org/mathscinet-getitem?mr=#1}{#1}}}
\providecommand{\bysame}{\leavevmode\hbox to3em{\hrulefill}\thinspace}
\providecommand{\MR}{\relax\ifhmode\unskip\space\fi MR }
\providecommand{\MRhref}[2]{%
  \href{http://www.ams.org/mathscinet-getitem?mr=#1}{#2}
}
\providecommand{\href}[2]{#2}

\end{document}